\documentclass[11pt,reqno]{article}

\usepackage{amsmath,amsthm,amsfonts,amssymb,latexsym,mathrsfs,url,enumerate}
\usepackage{graphicx,ifpdf,epstopdf}
\usepackage{comment}
\newtheorem{thm}{Theorem}[section]
\newtheorem{lem}[thm]{Lemma}
\newtheorem{coro}[thm]{Corollary}
\newtheorem{prop}[thm]{Proposition}
\newtheorem{prob}[thm]{Problem}
\newtheorem{conj}[thm]{Conjecture}

\theoremstyle{definition}

\numberwithin{equation}{section}

\begin{document}

\title{Partitioning digraphs with outdegree at least 4}

\author{Guanwu Liu\thanks{Partially supported by the China Scholarship Council.
E-mail address: liuguanwu@hotmail.com}\\
School of Mathematical Sciences\\ Dalian University of Technology\\
 Dalian 116024, P.R. China\\
and \\
School of Mathematics\\  Georgia Institute of Technology\\ Atlanta, GA 30332-0160, USA\\
\bigskip \\
 Xingxing Yu\thanks{Partially supported by NSF grant DMS 1954134.  E-mail address: yu@math.gatech.edu} \\
School of Mathematics\\
Georgia Institute of Technology\\
Atlanta, GA 30332-0160, USA}

\date{}

\maketitle

\begin{abstract}
Scott asked the question of determining $c_d$ such that if $D$ is a digraph with $m$ arcs and minimum outdegree $d\ge 2$ then $V(D)$ has a partition $V_1, V_2$
such that $\min\left\{e(V_1,V_2),e(V_2, V_1)\right\}\geq c_dm$, where $e(V_1,V_2)$ (respectively, $e(V_2,V_1)$)
is the number of arcs from $V_1$ to $V_2$ (respectively, from $V_2$ to $V_1$). Lee, Loh, and Sudakov showed that $c_2=1/6+o(1)$ and $c_3=1/5+o(1)$, and conjectured
that $c_d= \frac{d-1}{2(2d-1)}+o(1)$ for $d\ge 4$. In this paper, we show $c_4=3/14+o(1)$ and prove some partial results for $d\ge 5$.

\bigskip


\end{abstract}


\section{Introduction}

Judicious partitioning problems concern partitions of graphs and
hypergraphs that provide bounds for several parameters simultaneously,
while classical partitioning problems seek for partitions that optimize
a single parameter. For a graph $G$ and $A,B\subseteq V(G)$, we use $e(A,B)$ to denote the number of
edges in $G$ between $A$ and $B$, and we write $e(A):=e(A,A)$.
An example of a classical partitioning result is Edwards' theorem \cite{Edwards1973,Edwards1975} that
if $G$ is a graph with $m$ edges then $V(G)$ has a partition $V_1,V_2$
such that $e(V_1,V_2)\ge m/2+(\sqrt{2m+1/4}-1/2)/4$, and the
inequality is tight for complete graphs of odd order.
Bollob\'{a}s and Scott \cite{Bollobas1999} proved the following
judicious version of Edwards' result: The vertex set of any $m$-edge graph
has a bipartition $V_1, V_2$ such that
$e(V_1, V_2)\geq m/2+(\sqrt{2m+1/4}-1/2)/4$ and $\max\{e(V_1),e(V_2)\}\leq m/4+(\sqrt{2m+1/4}-1/2)/8$, and both bounds are tight for
complete graphs of odd order.

 Bollob\'{a}s and Scott \cite{Bollobas2002_1, Scott2005} initiated a systematic study of judicious partitioning problems, which has lead to
a large amount of research in this area, see, for instance,
\cite{Hou2020, Lee2013, Lee2014, MaJCTB, Ma2012, Ma2016, Xu2011,
  Xu2014, Xu2008, Zhu2009}.

Partitioning problems concerning digraphs (i.e., directed graphs)  may be more
difficult.
For a digraph $D$ and $A, B \subseteq V(D)$, we use $e(A, B)$ to denote the number of arcs in $D$ directed from $A$ to $B$ and write $e(A):=e(A,A)$.
Edwards' result above implies that every digraph $D$ with $m$
arcs has a vertex partition $V_1, V_2$ such that $e(V_1, V_2)\geq
m/4+(\sqrt{2m+1/4}-1/2)/8$, and the bound is tight for complete graphs of
odd order with an Eulerian orientation. On the other hand, Alon, Bollob\'{a}s, Gy\'{a}rf\'{a}s, Lehel, and Scott \cite{Alon2007} constructed digraphs
whose maximum directed cut  is $m/4 + O(m^{4/5})$.

A natural judicious version of Edwards' result is to bound both $e(V_1,V_2)$ and $e(V_2,V_1)$. Indeed,
 Scott \cite{Scott2005} asked the following question for digraphs
 without loops or parallel arcs in the same direction. (Throughout this paper, all digraphs
 have no loops or parallel arcs in the same direction.) Note that the {\it outdegree} of a vertex in a digraph is the number of
 arcs directed away from that vertex.

\begin{prob}[Scott \cite{Scott2005}] \label{scottprob}
What is the maximum constant $c_{d}$ such that every digraph $D$ with $m$ arcs and minimum outdegree $d\geq 2$ admits a bipartition $V(D)=V_1\cup V_2$
such that
\[
\min\left\{e(V_1, V_2), e(V_2, V_1)\right\}\geq c_{d}m?
\]
\end{prob}

The reason for  the requirement $d\ge 2$ in Problem~\ref{scottprob} is the following: Take the
star $K_{1,n-1}$ with $n\ge 4$, and add
a single edge between two vertices of degree 1. Orient the unique
triangle so that it becomes a directed cycle, and orient all other
edges so that they are directed towards the unique vertex of degree $n-1$.  This digraph has minimum outdegree 1, and  $e(V_1,V_2)\le 1$
for any bipartition $V_1,V_2$ of its vertex set with $V_1$
containing  the unique vertex of degree $n-1$.  Thus, $c_1=0$.

Lee, Loh, and Sudakov \cite{Lee2014} proved that  $c_2=1/6+o(1)$ and
$c_3=1/5+o(1)$, and they  made the following conjecture for
 $d\geq 4$.

\medskip

\begin{conj}[Lee, Loh, and Sudakov \cite{Lee2014}]\label{conj1}
Let $d$ be an integer satisfying $d\geq 4$. Every digraph $D$ with $m$ arcs and minimum outdegree at least $d$ admits a bipartition
$V(D)=V_1\cup V_2$ with
\[
\min\left\{e(V_1,V_2),e(V_2, V_1)\right\}\geq \left(\frac{d-1}{2(2d-1)}+o(1)\right)m.
\]
\end{conj}

The main term $\frac{d-1}{2(2d-1)}$ in Conjecture \ref{conj1} is best
possible,  because of examples constructed in \cite{Lee2014}
using copies of $K_{2d-1}$ and one copy of $K_{2d+1}$. Lee, Loh, and
Sudakov \cite{Lee2014} also noted that their tools for $d=2,3$ appear to be
insufficient for $d\ge 4$.  Hence, much effort has been devoted to studying variations of this problem, for instance,  by considering minimum
total degree conditions, see  \cite{hou, Hou2020, Hou2017, HWY2017}.  In this paper, we show that Conjecture~\ref{conj1} holds under certain natural conditions. In particular,
we prove Conjecture~\ref{conj1} for $d=4$.

\begin{thm}\label{thm1}
Every digraph $D$ with $m$ arcs and minimum outdegree at least $4$ admits a bipartition
$V(D)=V_1\cup V_2$ with
\[
\min\left\{e(V_1,V_2),e(V_2, V_1)\right\}\geq
\left(3/14+o(1)\right)m.
\]
\end{thm}

 In Section 2, we set up  notation  and list previous results needed in our proof of Theorem \ref{thm1}.
In Section 3, we describe and discuss our  approach for all $d$ and
obtain information in terms of ``huge''
vertices, vertices whose indegree and outdegree have a large gap. In Section 4, we show that Conjecture~\ref{conj1} holds
under some additional conditions on the number of huge vertices.
We complete the proof of Theorem \ref{thm1} in Section 5 and offer some concluding remarks in Section 6.

\section{Notation and lemmas}

We start with  notation and terminology that will be used in this paper.
Let $D$ be a digraph.  For $x\in V(D)$,
let $N_{D}^{+}(x)=\{y : xy\in E(D)\}$ and $N_{D}^{-}(x)=\{y : yx\in E(D)\}$. Then
$d_{D}^{+}(x):=|N_{D}^{+}(x)|$ and $d_{D}^{-}(x):=|N_{D}^{-}(x)|$ are the the $outdegree$ and $indegree$ of $x$,
respectively.
The $degree$ of $x\in V(D)$ is defined as $d_D(x)= d^{+}_D(x)+d^-_D(x)$. We use
$\Delta(D)= \max\{d_D(x) : x\in V(D)\}$ to
denote the $maximum \ degree$ of $D$.
For any $X\subseteq  V(D)$, the subgraph of $D$ induced by $X$ is denoted by $D[X]$.
We will often omit the subscript $D$ in the above notation when there is no
danger of confusion. It will be convenient to write $[k]$ for $\{1, \ldots,
k\}$, where $k$ is any positive integer.

Lee, Loh, and Sudakov \cite{Lee2014} proved that certain partial partitions of a digraph  may be extended
to a good partition of the entire digraph.

\begin{lem}[Lee, Loh, and Sudakov, \cite{Lee2014}]\label{lem1}
Let $D$ be a digraph with  $m$ arcs. Let $p$ be a real satisfying $p\in[0,1]$,
and let $\varepsilon>0$. Suppose that a subset $X\subseteq V$
and its partition $X=X_{1}\cup X_{2}$ are given, and let $Y=V\setminus X$.
Further suppose that $\max_{y\in Y}d(y)\le \varepsilon^{2}m/{4}$.
Then there exists a bipartition $V(D)=V_{1}\cup V_{2}$ with $X_i\subseteq V_i$ for $i\in [2]$ such that
$$e(V_{1},V_{2})  \ge e(X_{1},X_{2})+(1-p)\cdot e(X_{1}, Y)+p\cdot e(Y, X_{2})+p(1-p)\cdot e(Y)-\varepsilon m,$$
$$e(V_{2},V_{1})  \ge e(X_{2},X_{1})+p\cdot e(X_{2}, Y)+(1-p)\cdot e(Y, X_{1})+p(1-p)\cdot e(Y)-\varepsilon m.$$
\end{lem}

By applying  Lemma \ref{lem1}  with $p=1/2$ and $X_1=X_2=\emptyset$ and by noting that $d_D(v)\le 2|V(D)|$, we obtain the following.

\begin{coro}[Lee, Loh, and Sudakov \cite{Lee2014}]\label{lem2}
Let $D$ be a digraph with $n$ vertices and $m$ arcs. For any $\varepsilon>0$, if $m\geq8n/\varepsilon^{2}$
or $\Delta(D)\leq\varepsilon^{2}m/4$, then $D$ admits a bipartition $V(D)=V_{1}\cup V_{2}$ with
$\min\{e(V_{1},V_{2}),e(V_{2},V_{1})\}\geq m/4-\varepsilon m$.
\end{coro}

From Corollary \ref{lem2} we see that if the maximum degree of a digraph $D$ is not too large,
then $V(D)$ admits a partition $V_1, V_2$ such that both $e(V_1,V_2)$ and $e(V_2,V_1)$ are close to $m/4$.
We will see that the vertices causing problems for obtaining the desired partition in Conjecture~\ref{conj1} are those
whose outdegree and indegree differ significantly. Hence, for $x\in V(D)$, let
 $$s^+(x):=d^+(x)-d^-(x), s^-(x):=d^-(x)-d^+(x), \mbox{ and }
 s(x):=\max\{s^+(x),s^-(x)\}.$$
Note that $d(x)-s(x)$ is an even integer, and we often write $$2b=\sum_{x\in X}\big(d(x)-s(x)\big).$$

To study those vertices $x$ with large $s(x)$, we need the concept of the gap of a partition.
Let $D$ be a digraph and let $X,Y$ be a partition of $V(D)$. For each partition $X_1,X_2$  of $X$,
the {\it gap} of $X_1,X_2$ is defined as
$$\theta(X_1,X_2)=  \Big(e(X_{1},Y)+e(Y,X_{2})\Big)-\Big(e(X_{2},Y)
+e(Y,X_{1})\Big).$$
The {\it huge} vertices of $D$ with respect to the partition $X,Y$ are
the vertices $x$ such that $$s(x) \ge \min\{|\theta(X_1,X_2)| :
X_1,X_2  \mbox{ is a partition of } X\}.$$

Let $D$ be a digraph, $X,Y$ a partition of $V(D)$, and 
$X_1,X_2$ a partition of $X$.  For convenience, let
$m_f(X_1,X_2)=e(X_1,Y)+e(Y,X_2)$ and $m_b(X_1,X_2)=e(X_2,Y)+e(Y,X_1)$;
so
$$\theta(X_1,X_2)=m_f(X_1,X_2)-m_b(X_1,X_2).$$
Note that $$\theta(X_1,X_2)   =  \Big(e(X_{1},Y)-e(Y,X_1)\Big)-\Big(e(X_{2},Y) -e(Y,X_{2})\Big).$$
Thus, if $e(X)=0$ then
\begin{equation}\label{theta}
\begin{aligned}
\theta(X_1,X_2)   =   \left(\sum_{x\in X_1}s^+(x)\right)-\left( \sum_{x\in X_2}s^+(x)\right)
\end{aligned}
\end{equation}
For any $x\in X$, we say that $x$ is
\begin{itemize}
\item [] {\it $(X_1,X_2)$-forward}  if $x\in X_1$ and $s^+(x)>0$, or $x\in X_2$ and $s^-(x)>0$, and
\item [] {\it $(X_1,X_2)$-backward}  if  $x\in X_1$ and $s^-(x)>0$, or $x\in X_2$ and $s^+(x)>0$.
\end{itemize}
Let $X_f:=\{x\in X: x \mbox{ is $(X_1,X_2)$-forward}\}$ and $X_b:=\{x\in X: x \mbox{ is $(X_1,X_2)$-backward}\}$.
By (\ref{theta}), if $e(X)=0$ then
\begin{equation}\label{theta2}
\theta(X_1,X_2)  =\sum_{x\in X_f}s(x) - \sum_{x\in X_b}s(x).
\end{equation}

We will need the following result from \cite{HWY2017}.

\begin{lem}[Hou, Wu, and Yan \cite{HWY2017}]\label{lem5}
Let $D$ be a digraph and $V(D)=X\cup Y$ be a partition of $D$ with $e(X)=0$.
Let $X=X_{1}\cup X_{2}$ be a partition of $X$ that minimizes
$|\theta(X_1,X_2)|$ among all partitions of $X$. Then
\begin{itemize}
\item [$(1)$]  $|\theta(X_1,X_2)|\leq|Y|$,  and
\item [$(2)$] $g:=\sum_{\{v\in X : s(v)<|\theta(X_1,X_2)|\}}s(v) \leq |Y|-|\theta(X_1,X_2)|$.
\end{itemize}
\end{lem}

\section{Properties of partitions with minimum gap}

In this section, we explore  the probabilistic approach
used by Lee, Loh, and Sodakov \cite{Lee2013, Lee2014}.  In particular, we investigate digraph partitions whose gaps have minimum absolute value.
We will prove several properties about gaps and huge vertices, by
considering various ways to partition the set of huge vertices. Those properties may be useful for the
eventual  resolution of Conjecture~\ref{conj1}.


\begin{lem}\label{property1}
Let $D$ be a digraph with $m$ arcs and minimum outdegree $d\ge 4$, and
let $X,Y$ be a partition of $V(D)$ with $e(X)=0$. Let
$\theta=\min\{|\theta(X_1,X_2)|: X_1,X_2 \mbox{ is a partition of } X\}$, and let $X'=\{x\in X: s(x)\ge \theta\}$.
Let $\varepsilon >0$ such that $\max_{y\in Y}d(y) \le \varepsilon^2m/4$.
Then there exists a partition
         $V_1,V_2$ of $V(D)$ such that $\min\{e(V_1,V_2), e(V_2,V_1)\}\ge \left(
  \frac{d-1}{2(2d-1)}-\varepsilon\right) m$, or the following statements hold:
   \begin{itemize}
       \item [$(1)$]  $\theta >m/(2d-1)$;
       \item [$(2)$] $|X'|$ is an odd integer;
       \item [$(3)$] letting $g=\sum_{x\in X\setminus X'}s(x)$ and
         $X'=\{v_1, \ldots, v_{2k+1}\}$ such that $s(v_1)\ge s(v_2)\ge
         \ldots \ge s(v_{2k+1})$, we have 
$\sum_{j=k+1}^{2k+1}s(v_j)-\sum_{j=1}^{k}s(v_j)\geq g+\theta.$
   \end{itemize}
\end{lem}

\begin{proof}
 Suppose, for any partition $V_1,V_2$ of $V(D)$, $\min \{e(V_1,V_2), e(V_2,V_1)\}< \left( \frac{d-1}{2(2d-1)}-\varepsilon\right) m$. We show that (1), (2), and (3)
hold.

First, we prove (1).   Let $m_1=e(X,Y)+e(Y,X)$ and $m_2=e(Y)$. Thus,
$m=m_1+m_2$,  as $e(X)=0$. Let
$(X_1,X_2)$ be a partition of $X$ such that  $\theta(X_1,X_2)=\theta$.
Applying Lemma \ref{lem1} with $p=1/2$, 
there is a bipartition $V_{1},V_{2}$ of $V(D)$ such that
$X_{i}\subseteq  V_{i}$ for $i\in [2]$, and
\begin{align*}
& \min\{e(V_{1},V_{2}),e(V_{2},V_{1})\}
\\ \geq &\frac{1}{2}\min\{e(X_{1},Y)+e(Y,X_{2}),e(X_{2},Y) +e(Y,X_{1})\}+\frac{e(Y)}{4}-\varepsilon m
\\ = &\frac{m_{1}-\theta}{4}+\frac{m_{2}}{4}-\varepsilon m
\\ =&\frac{m-\theta}{4}-\varepsilon m.
\end{align*}
If $\theta \le  m/(2d-1)$ then $(m-\theta)/4\ge (d-1)m/\big(2(2d-1)\big)$; so $\min \{e(V_1,V_2), e(V_2,V_1)\}\ge \left( \frac{d-1}{2(2d-1)}-\varepsilon\right) m$, a contradiction. Thus, $\theta >  m/(2d-1)$, and (1) holds.

\medskip

Let $X=\{v_1, \ldots, v_{|X|}\}$ such that $s(v_1)\ge s(v_2)\ldots \ge s(v_{|X|})$.
To prove (2), let us assume $|X'|=2k$ for some non-negative integer $k$.  Then $X'=\{v_1, \ldots, v_{2k}\}$.

First, suppose $k=0$. Then $s(v_1)<\theta$.  Let $X_1^*,X_2^*$ be the partition of  $X$ such that $v_1,v_3,\cdots, v_{2p-1}$ are $(X_1^*,X_2^*)$-forward, where $p=\lceil |X|/2 \rceil$,  and all other vertices are $(X_1^*,X_2^*)$-backward. 
If $|X|$ is even then, by \eqref{theta2},
\[
\left|\theta(X_1^*,X_2^*) \right|
                               = \left|s(v_1) -\sum_{i=1}^{p-1} \big(s(v_{2i})-s(v_{2i+1})\big) -s(v_{|X|})\right|
                                \le  s(v_1)
                                <\theta,
\]
a contradiction. If $|X|$ is odd then, by \eqref{theta2},
\[
   \left|\theta(X_1^*,X_2^*)  \right| = \left|s(v_1)-\sum_{i=1}^{p-1}\big(s(v_{2i}) -s(v_{2i+1})\big)\right| \le s(v_1)< \theta,
\]
a contradiction.

Now suppose $k>0$.  Then $s(v_{2k})\ge \theta$. Let $X_1^*,X_2^*$ be the partition of  $X$ such that
$v_1,v_3,\cdots, v_{2k-1}$ are $(X_1^*,X_2^*)$-forward, and all other vertices in $X$ are $(X_1^*,X_2^*)$-backward.
Then  by \eqref{theta2}, $\left|\theta(X_1^*,X_2^*)\right|=\left|\sum_{i=1}^ks(v_{2i-1}) -\big(\sum_{i=1}^ks(v_{2i})+g\big)\right|$.
Note that
\begin{align*}
     \sum_{i=1}^ks(v_{2i-1}) -\left(\sum_{i=1}^ks(v_{2i})+g\right)
& = s(v_1)-\sum_{i=1}^{k-1}\Big(s(v_{2i})-s(v_{2i+1})\Big) -s(v_{2k})-g\\
& \le s(v_1)-s(v_{2k})\\
& \le  |Y|-\theta,
\end{align*}
and, since $g\le |Y|-\theta$ (by Lemma~\ref{lem5}),
\begin{align*}
 \left(\sum_{i=1}^ks(v_{2i})+g\right)-  \sum_{i=1}^ks(v_{2i-1})
 & =  \sum_{i=1}^{k}\Big(s(v_{2i})-s(v_{2i-1})\Big)+g
  \le g  \le |Y|-\theta.
\end{align*}
Hence, $\left|\theta(X_1^*,X_2^*)\right|\le |Y|-\theta$.
Because $\theta >m/(2d-1)$ (by (1)) and $m\ge d|V(D)|$, we see that $\theta > |V(D)|/2\ge |Y|/2$. Thus, $\left|\theta(X_1^*,X_2^*)\right|\le |Y|-\theta<\theta$,
a contradiction. Thus, $|X'|$ must be odd, and we have (2).

\medskip

By (2), let $X':=\{v_1, \ldots, v_{2k+1}\}$ for some $k\ge 0$.
Recall that $d(x)-s(x)$ is an even integer for all $x\in X$, and we write $2b=\sum_{x\in X}\big(d(x)-s(x)\big)$.
To prove (3), we consider the partition $X_1^1,X_2^1$ of $X$ such that $\{v_{1},v_{3},\cdots,v_{2k-1}\}\cup (X\setminus X')$ is the set of $(X^1_1,X^1_2)$-forward
vertices, and $\{v_{2},v_{4},\cdots,v_{2k},v_{2k+1}\}$ is the set of $(X^1_1,X^2_2)$-backward vertices.
Then  $m_f(X^1_1,X^1_2)=\sum_{j=1}^{k}s(v_{2j-1})+g+b$ and $m_b(X^1_1,X^1_2)=\sum_{j=1}^{k}s(v_{2j})+s(v_{2k+1})+b$.
Note that
\begin{align*}
 \theta(X_1^1,X_2^1) =& m_f(X^1_1,X^1_2)-m_b(X_1^1,X_2^1)\\
 =&\sum_{j=1}^{k}s(v_{2j-1})+g-\sum_{j=1}^{k}s(v_{2j})-s(v_{2k+1})\\
   =  & \big(s(v_1)-s(v_{2k})-s(v_{2k+1})\big)+\sum_{j=1}^{k-1}\big(s(v_{2j+1})-s(v_{2j})\big)+g\\
   \le & s(v_1)-s(v_{2k})-s(v_{2k+1})+|Y|-\theta  \quad (\mbox{since $g\le |Y|-\theta$ by Lemma~\ref{lem5}})\\
   \le & s(v_1)+|Y|-3\theta\\
   \le  & 2|V(D)|-3\theta \\
  < & \theta \quad (\mbox{as $\theta >|V(D)|/2$ because $m\ge d|V(D)|$ and by (1)}).
\end{align*}
Thus, since $|m_f(X^1_1,X^1_2)-m_b(X^1_1,X^1_2)|=\left|\theta(X^1_1,X^1_2)\right|\ge \theta$, we have
$$m_b(X^1_1,X^1_2)-m_f(X^1_1,X^1_2)=\sum_{j=1}^{k}s(v_{2j})+s(v_{2k+1})-\sum_{j=1}^{k}s(v_{2j-1})-g\geq \theta.$$

Now consider the partition $X_1^2,X_2^2$ of $X$ such that $X_1^2=\left(X^1_1\setminus \{v_{2k-1}\}\right)\cup \{v_{2}\}$ and $X_2^2=\left(X_2^1\setminus \{v_2\}\right)\cup \{v_{2k-1}\}$. Then $m_f(X_1^2,X_2^2)=m_f(X_1^1,X_2^1)-s(v_{2k-1})+s(v_2)$ and $m_b(X_1^2,X_2^2)=m_b(X_1^1,X_2^1)-s(v_2)+s(v_{2k-1})$.
Hence, $$m_b(X_1^2,X_2^2)-m_f(X_1^2,X_2^2)=\Big(m_b(X_1^1,X_2^1)-m_f(X_1^1,X_2^1)\Big)-2\Big(s(v_{2})-s(v_{2k-1})\Big),$$
which implies that
$$m_b(X_1^2,X_2^2)-m_f(X_1^2,X_2^2)\ge \theta-2(|Y|-\theta)>\theta-2\theta=-\theta.$$
Therefore, since $\left|m_b(X_1^2,X_2^2)-m_f(X_1^2,X_2^2)\right|=|\theta(X_1^2,X_2^2)|\ge \theta$, we see that
$$m_b(X_1^2,X_2^2)-m_f(X_1^2,X_2^2)\ge \theta.$$

Repeating the same argument by exchanging the sides for $v_{2(k-i)+1}$ and $v_{2i}$, one step at a time in the order $i=2, \ldots, \lfloor k/2 \rfloor$, we arrive at the partition
$X_1^{\lfloor k/2 \rfloor},X_2^{\lfloor k/2 \rfloor}$ of $X$, such that $\{v_1, v_2, \ldots, v_k\}\cup (X\setminus X')$ is the set of $(X_1^{\lfloor k/2 \rfloor},X_2^{\lfloor k/2 \rfloor})$-forward vertices, $\{v_{k+1}, v_{k+2}, \ldots, v_{2k+1}\}$
is the set of $(X_1^{\lfloor k/2 \rfloor},X_2^{\lfloor k/2 \rfloor})$-backward vertices, and
 $$m_b(X_1^{\lfloor k/2 \rfloor},X_2^{\lfloor k/2 \rfloor})-m_f(X_1^{\lfloor k/2 \rfloor},X_2^{\lfloor k/2 \rfloor})\ge \theta.$$
On the other hand, we have, by \eqref{theta2}, that
\[
 m_b(X_1^{\lfloor k/2 \rfloor},X_2^{\lfloor k/2 \rfloor})-m_f(X_1^{\lfloor k/2 \rfloor},X_2^{\lfloor k/2 \rfloor}) =\sum_{j=k+1}^{2k+1}s(v_j)-\sum_{j=1}^{k}s(v_j)-g.
\]
Hence, (3) holds.
\end{proof}

\begin{lem}\label{property2}

Let $D$ be a digraph with $m$ arcs and minimum outdegree $d\ge 4$, and
let $X,Y$ be a partition of $V(D)$ with $e(X)=0$. Let
$\theta=\min\{|\theta(X_1,X_2)|: X_1, X_2 \mbox{ is a partition 
  of } X\}$,  $X'=\{x\in X: s(x)\ge \theta\}$, 
$g=\sum_{x\in X\setminus X'}s(x)$, and  $2b=\sum_{x\in
  X}\big(d(x)-s(x)\big)$. 
Let $\varepsilon>0$ and assume that $\max_{y\in Y}d(y)\le \varepsilon^2m/4$.
Then there exists a partition $V_1,V_2$ of $V(D)$ such that
$\min\{e(V_1,V_2),e(V_2,V_1)\}\ge 
\left( \frac{d-1}{2(2d-1)}-\varepsilon\right) m$;  or $|X'|$ is odd
and if we let $X'=\{v_1, \ldots, v_{2k+1}\}$ such that
$s(v_1)\ge s(v_2)\ge \ldots \ge s(v_{2k+1})$ and write $\Delta_j=s(v_j)$ for $j\in [2k+1]$  then
   \begin{itemize}
       \item [$(1)$]  $d\left(\sum_{j=1}^{k}\Delta_{j}+g\right)-(d-1)\sum_{j=k+1}^{2k+1}\Delta_{j} +b+e(Y)/2<0$,
       \item [$(2)$]  $b>  \frac{d^2+2d-1}{d-1}\sum_{j=k}^{2k-1}\Delta_{j}-d\left(\sum_{j=1}^{k-1}\Delta_j+
     \Delta_{2k}+\Delta_{2k+1}\right)+(d-1)g +\frac{d-1}{2d}e(Y)$,
       \item [$(3)$]  $b< \frac{2(2d-1)(k+1)}{3d-1}|V(D)|+\frac{d^2-5d+2}{3d-1}\left(\sum_{j=1}^{k-1}\Delta_j+\Delta_{2k}+\Delta_{2k+1}\right)
   -\frac{d^2+2d-1}{3d-1}\sum_{j=k}^{2k-1}\Delta_{j} +\frac{d(d-1)}{3d-1}g
  -\frac{(d-1)^2}{2d(3d-1)} e(Y)$, and
      \item [$(4)$] when $d=4$ and $k=1$, $2\Delta_2+2\Delta_3-3\Delta_1-3g-b+3e(Y)/14<0$, or both
        $6\Delta_1-3\Delta_2-3\Delta_3+2g-b+3e(Y)/14<0$ and  $6\Delta_3-3\Delta_1-3\Delta_2-3g+b/3+3e(Y)/14<0$.

   \end{itemize}

\end{lem}

\begin{proof}
For convenience, we introduce two functions which we will use to compare $\min \{e(V_1,V_2)$, $e(V_2, V_1)\}$ with
$\left( \frac{d-1}{2(2d-1)}-\varepsilon\right) m$ for any partition
$V_1,V_2$ of $V(D)$.  For any partition
$X_1,X_2$ of $X$, let $z(X_1,X_2)=e(X_1,Y)$ and
$z'(X_1,X_2)=e(Y,X_2)$; so
$m_f(X_1,X_2)=z(X_1,X_2)+z'(X_1,X_2)$.
Let $m_1:=e(X,Y)+e(Y,X)$,
$m_2:=e(Y)$, and  $$\ell(p,X_1,X_2):= (d-1)\sum_{j=1}^{2k+1}\Delta_j+(d-1)g + (2d-2)b-\Big(2(2d-1)p(1-p)-(d-1)\Big)m_2.$$
Define
\begin{itemize}
   \item [] $f(p,X_1,X_2)= 2(1-p)(2d-1)z(X_1,X_2)+2p(2d-1)z'(X_1,X_2) -\ell(p,X_1,X_2)$, and
   \item [] $h(p,X_1,X_2)= 2p(2d-1)\Big(m_1-z(X_1,X_2)-z'(X_1,X_2)\Big)-\ell(p,X_1,X_2)$.
\end{itemize}

By Lemma \ref{lem1}, for any $0\leq p\leq 1$, there is a partition $V(D)=V_1\cup V_2$ such that
$X_i\subseteq V_i$ for $i\in [2]$, and
\begin{equation}\label{eq:1}
\left\{
\begin{aligned}
e(V_{1},V_{2})  \ge (1-p)\cdot e(X_1,Y)+p\cdot e(Y,X_2)+p(1-p)\cdot e(Y)-\varepsilon m, \\
e(V_{2},V_{1})  \ge p\cdot e(X_{2}, Y)+(1-p)\cdot e(Y,X_1)+p(1-p)\cdot e(Y)-\varepsilon m.
\end{aligned}
\right.
\end{equation}
Without loss of generality, we may assume $p\leq 1-p$; so $p\le
1/2$. Then, from \eqref{eq:1}, we have
\begin{equation}\label{eq6}
\left\{
\begin{aligned}
 e(V_{1},V_{2}) & \ge (1-p)z(X_1,X_2)+p z'(X_1,X_2)+p(1-p) m_2-\varepsilon m, \\
 e(V_{2},V_{1}) & \ge  p\big(m_1-z(X_1,X_2)-z'(X_1,X_2)\big)+p(1-p)m_2-\varepsilon m.
\end{aligned}
\right.
\end{equation}

Note that
\begin{equation}\label{eq5}
m=m_1+m_2=\sum_{j=1}^{2k+1}\Delta_{j}+g+2b+m_2.
\end{equation}
By \eqref{eq6} and \eqref{eq5}, we have
\begin{align*}
& e(V_{1},V_{2})-\left(\frac{d-1}{2(2d-1)}m-\varepsilon m\right)\\
\geq & (1-p)z(X_1,X_2)+p z'(X_1,X_2)+p(1-p) m_2-\frac{d-1}{2(2d-1)}\left(\sum_{j=1}^{2k+1}\Delta_{j}+g+2b+m_2\right)\\
= & \frac{1}{2(2d-1)} f(p,X_1,X_2),
\end{align*}
and
\begin{align*}
& e(V_{2},V_{1})-\left(\frac{d-1}{2(2d-1)}m-\varepsilon m\right)\\
\geq & p\Big(m_1-z(X_1,X_2)-z'(X_1,X_2)\Big)+p(1-p)m_2-\frac{d-1}{2(2d-1)}\left(\sum_{j=1}^{2k+1}\Delta_{j}+g+2b+m_2\right)\\
= & \frac{1}{2(2d-1)} h(p,X_1,X_2).
\end{align*}

If $f(p,X_1,X_2)\ge 0$ and $h(p,X_1,X_2)\ge 0$ for some choice of $p,X_1,X_2$, we see that the there is a partition $V_1, V_2$ of $V(D)$ such that
for $i\in [2]$, $X_i\subseteq V_i$ and
$e(V_i,V_{3-i})\ge \left( \frac{d-1}{2(2d-1)}-\varepsilon\right) m$.
Hence, we may assume that
\begin{equation}\label{eq:a}
f(p,X_1,X_2)<0 \mbox{ or } h(p,X_1,X_2)<0 \mbox{  for any choice of } p,X_1,X_2.
\end{equation}

To see (1), we consider the partition $X_1^1,X_2^1$ of $X$ such that $\{v_{1},v_{2},\ldots,v_{k}\}\cup (X\setminus X')$ is the set of $(X_1^1,X_2^1)$-forward vertices,
 and  $\{v_{k+1},v_{k+2},\ldots,v_{2k+1}\}$ is the set of $(X_1^1,X_2^1)$-backward vertices.  Then $z(X_1^1,X_2^1)+z'(X_1^1,X_2^1)=m_f(X_1^1,X_2^1)=\sum_{j=1}^{k}\Delta_{j}+g+b$ and
$m_b(X_1^1,X_2^1)=\sum_{j=k+1}^{2k+1}\Delta_{j}+b$.  Setting $p=1/2$, it follows from a simple calculation that
\begin{equation*}\label{eq42}
\left\{
\begin{aligned}
& f(1/2,X_1^1,X_2^1)=d\left(\sum_{j=1}^{k}\Delta_{j}+g\right)-(d-1)\sum_{j=k+1}^{2k+1}\Delta_{j}
+b+m_2/2,\\
& h(1/2,X_1^1,X_2^1)=d\sum_{j=k+1}^{2k+1}\Delta_{j}-(d-1)\left(\sum_{j=1}^{k}\Delta_{j}+g\right)+b+m_2/2.
\end{aligned}
\right.
\end{equation*}
By (3) of Lemma~\ref{property1}, we may assume $ h(1/2,X_1^1,X_2^1)>0$; so by \eqref{eq:a},  we have $f(1/2,X_1^1,X_2^1)<0$.
Thus, (1) holds.

\medskip

For  (2) and (3), we note that at least $k$ members of  $\{s^+(v_1),s^+(v_2),\cdots,s^+(v_{2k-1})\}$
have the same sign; so we may assume that
$s^+(v_{j_1})$, $s^+(v_{j_2})$, $\cdots$, $s^+(v_{j_k})$ are positive,
where $1\leq j_1<j_2<\cdots<j_{k}\leq 2k-1$.
\medskip

To prove (2), let $X_1^2,X_2^2$ be the partition of $X$ such that $\{v_{j_1},v_{j_2},\cdots,v_{j_k}\}\cup (X\setminus X')$ is the set of $(X_1^2,X_2^2)$-forward
vertices, and all other vertices in $X$ are $(X_1^2, X_2^2)$-backward.
Then  $$m_f(X_1^2,X_2^2)=\sum_{i=1}^{k}\Delta_{j_i}+g+b$$
and
$$m_b(X_1^2,X_2^2)=\sum_{j=1}^{2k+1}\Delta_j-\sum_{i=1}^{k}\Delta_{j_i}+b\geq \sum_{j=k+1}^{2k+1}\Delta_j+b.$$
Note that $z(X_1^2,X_2^2)=e(X_1^2,Y)\geq \sum_{i=1}^{k}\Delta_{j_i}$
and $z'(X_1^2,X_2^2)=e(Y,X_2^2)=m_f(X_1^2,X_2^2)-e(X_1^2,Y)$. Setting  $p=(d-1)/(2d)$, we see that
\begin{align*}
& f\left(\frac{d-1}{2d},X_1^2,X_2^2\right) \\
\ge & \frac{(d+1)(2d-1)}{d}\sum_{i=1}^{k}\Delta_{j_i}  +\frac{(d-1)(2d-1)}{d}(g+b)- (d-1)\sum_{j=1}^{2k+1}\Delta_j\\
       & -(d-1)g-(2d-2)b+\frac{(d-1)^2}{2d^2}m_2\\
 = & \frac{d-1}{d} \left( \frac{d^2+2d-1}{d-1}\sum_{i=1}^{k}\Delta_{j_i}-d\Big(\sum_{j=1}^{2k+1}\Delta_{j}-\sum_{i=1}^{k}\Delta_{j_i}\Big) +(d-1)g-b+\frac{d-1}{2d}m_2 \right) \\
 \geq & \frac{d-1}{d} \left(\frac{d^2+2d-1}{d-1}\sum_{j=k}^{2k-1}\Delta_{j}-d\Big(\sum_{j=1}^{k-1}\Delta_{j}
 +\Delta_{2k}+\Delta_{2k+1}\Big) +(d-1)g-b+\frac{d-1}{2d}m_2 \right),
\end{align*}
and
\begin{align*}
& h\left(\frac{d-1}{2d},X_1^2,X_2^2\right)\\
\geq & \frac{(d-1)(2d-1)}{d}\left(\sum_{j=k+1}^{2k+1}\Delta_j+b\right)-(d-1)\sum_{j=1}^{2k+1}\Delta_j-(d-1)g
  -(2d-2)b+\frac{(d-1)^2}{2d^2}m_2\\
 = & \frac{d-1}{d}\left((d-1)\sum_{j=k+1}^{2k+1}\Delta_{j}-d\Big(\sum_{j=1}^{k}\Delta_{j}+g\Big)-b+\frac{d-1}{2d}m_2\right).
\end{align*}
By (1),  $h\left(\frac{d-1}{2d},X_1^2,X_2^2\right)>0$. So  $f\left(\frac{d-1}{2d},X_1^2,X_2^2\right)<0$ by \eqref{eq:a}. Hence, (2) holds.

\medskip

To prove (3), consider the partition $X_1^3,X_2^3$ of $X$ such that $\{v_{j_1},v_{j_2},\cdots,v_{j_k}\}\cup (X\setminus X')\subseteq X_1^3$,
and the  vertices in $X'\setminus \{v_{j_1}, \ldots, v_{j_k}\}$ are $(X_1^3, X_2^3)$-backward.
Since $s^+(v_{j_i})>0$ for $i\in [k]$, the vertices  $v_{j_1},v_{j_2},\cdots,v_{j_k}$  are $(X_1^3, X_2^3)$-forward.
Hence,   $$m_f(X_1^3,X_2^3)\ge \sum_{i=1}^{k}\Delta_{j_i}+b$$ and
$$m_b(X_1^3,X_2^3)\ge \sum_{j=1}^{2k+1}\Delta_j-\sum_{i=1}^{k}\Delta_{j_i}+b\geq \sum_{j=k+1}^{2k+1}\Delta_j+b.$$
Let $n:=|V(D)|$. Note that $X_2^3\subseteq X'\setminus \{v_{j_1},
\ldots, v_{j_k}\}$; so 
$$z'(X_1^3,X_2^3)=e(Y,X_2^3)\le (k+1)n- \Big(\sum_{j=1}^{2k+1}\Delta_{j}-\sum_{i=1}^{k}\Delta_{j_i}\Big).$$
Hence,
$$z(X_1^3,X_2^3)=e(X_1^3,Y)\geq \sum_{i=1}^{k}\Delta_{j_i}+b-\left((k+1)n-
\Big(\sum_{j=1}^{2k+1}\Delta_{j}-\sum_{i=1}^{k}\Delta_{j_i}\Big)\right).$$
Setting  $p=(d-1)/(2d)$, we have
\begin{align*}
 & f\left(\frac{d-1}{2d},X_1^3,X_2^3\right) \\
 \ge & \frac{(d+1)(2d-1)}{d}\left[\sum_{i=1}^{k}\Delta_{j_i}+b-\Big((k+1)n-
 \big(\sum_{j=1}^{2k+1}\Delta_{j}-\sum_{i=1}^{k}\Delta_{j_i}\big)\Big)\right]\\
 & +\frac{(d-1)(2d-1)}{d}\left((k+1)n-
 \Big(\sum_{j=1}^{2k+1}\Delta_{j}-\sum_{i=1}^{k}\Delta_{j_i}\Big)\right)-(d-1)\sum_{j=1}^{2k+1}\Delta_j\\
&  -(d-1)g-(2d-2)b+\frac{(d-1)^2}{2d^2}m_2\\
 = & \frac{3d-1}{d}\Big(\frac{d^2+2d-1}{3d-1}\sum_{i=1}^{k}\Delta_{j_i}-\frac{d^2-5d+2}{3d-1}\big(\sum_{j=1}^{2k+1}\Delta_{j}-
\sum_{i=1}^{k}\Delta_{j_i}\big)-\frac{2(2d-1)(k+1)}{3d-1}n \\
 & -\frac{d(d-1)}{3d-1}g+b+\frac{(d-1)^2}{2d(3d-1)}m_2\Big) \\
 \geq &\frac{3d-1}{d}\Big( \frac{d^2+2d-1}{3d-1}\sum_{j=k}^{2k-1}\Delta_{j}-\frac{d^2-5d+2}{3d-1}\big(\sum_{j=1}^{k-1}\Delta_{j}
 +\Delta_{2k}+\Delta_{2k+1}\big)-\frac{2(2d-1)(k+1)}{3d-1}n\\
 & -\frac{d(d-1)}{3d-1}g+b+\frac{(d-1)^2}{2d(3d-1)}m_2\Big),
\end{align*}
and
\begin{align*}
& h\left(\frac{d-1}{2d},X_1^3,X_2^3\right)\\
\geq & \frac{(d-1)(2d-1)}{d}\left(\sum_{j=k+1}^{2k+1}\Delta_j+b\right)-(d-1)\sum_{j=1}^{2k+1}\Delta_j-(d-1)g-(2d-2)b+\frac{(d-1)^2}{2d^2}m_2\\
 = & \frac{d-1}{d}\left((d-1)\sum_{j=k+1}^{2k+1}\Delta_{j}-d\Big(\sum_{j=1}^{k}\Delta_{j}+g\Big)-b+\frac{d-1}{2d}m_2\right).
\end{align*}
By (1), $h\left(\frac{d-1}{2d},X_1^3,X_2^3\right)>0$. So
$f(\frac{d-1}{2d},X_1^3,X_2^3)<0$ by \eqref{eq:a}. Hence, (3) holds.

\medskip

Now we prove (4); so assume $d=4$ and $k=1$.
First, let $X_1^4,X_2^4$ be the partition of $X$ such that $\{v_1\}\cup (X\setminus X')$ is the set of $(X_1^4,X_2^4)$-forward
vertices, and $v_2, v_3$ are $(X_1^4, X_2^4)$-backward.
Then  $m_f(X_1^4,X_2^4)=\Delta_1+g+b$ and
$m_b(X_1^4,X_2^4)=\Delta_2+\Delta_3+b.$
Also, we have $e(X_1^4,Y)\geq \Delta_1$. Setting  $p=5/14$, we see that
\begin{align*}
  h\left(5/14,X_1^4,X_2^4\right)
 \ge & 5(\Delta_2+\Delta_3+b)-3(\Delta_1+\Delta_2+\Delta_3)-3g-6b+3m_2/14\\
  = & 2\Delta_2+2\Delta_3-3\Delta_1-3g-b+3m_2/14\\
\end{align*}
and
\begin{align*}
  f\left(5/14,X_1^4,X_2^4\right)
  \ge & 9\Delta_1+5(g+b)-3(\Delta_1+\Delta_2+\Delta_3)-3g-6b+3m_2/14\\
   = & 6\Delta_1-3\Delta_2-3\Delta_3+2g-b+3m_2/14.
 \end{align*}
Thus, we have  $2\Delta_2+2\Delta_3-3\Delta_1-3g-b+3m_2/14<0$ or
$6\Delta_1-3\Delta_2-3\Delta_3+2g-b+3m_2/14<0$.

Next, we choose some $i\in [3]$ such that the number of arcs from $v_i$ to $Y$
counted in $b$ is maximum. Consider the partition $X_1^5,X_2^5$ of $X$
such that $\{v_i\}\cup (X\setminus X')\subseteq X_1^5$, and the
vertices in $X'\setminus \{v_i\}$ are $(X_1^5, X_2^5)$-backward.
Then, clearly, $m_f(X_1^5,X_2^5)\geq \Delta_i+b\geq \Delta_3+b$ and
$m_b(X_1^5,X_2^5)\geq \Delta_2+\Delta_3+b.$
Also, we have $e(X_1^5,Y)\geq (\Delta_i+b)-2b/3\geq \Delta_3+b/3$. Setting  $p=5/14$, we see that
\begin{align*}
  h\left(5/14,X_1^5,X_2^5\right)
&\ge 5(\Delta_2+\Delta_3+b)-3(\Delta_1+\Delta_2+\Delta_3)-3g-6b+3m_2/14\\
&  \ge 2\Delta_2+2\Delta_3-3\Delta_1-3g-b+3m_2/14,
\end{align*}
and
\begin{align*}
  f\left(5/14,X_1^5,X_2^5\right)
  & \ge 5(\Delta_3+b)+4(\Delta_3+b/3)-3(\Delta_1+\Delta_2+\Delta_3)-3g-6b+3m_2/14\\
 & =  6\Delta_3-3\Delta_1-3\Delta_2-3g+b/3+3m_2/14.
\end{align*}
Thus,  $2\Delta_2+2\Delta_3-3\Delta_1-3g-b+3m_2/14<0$ or
$6\Delta_3-3\Delta_1-3\Delta_2-3g+b/3+3m_2/14<0$.   This
completes the proof of (4).
\end{proof}

\section{Huge vertices}

In this section, we show that if $V(D)$ has a partition $X,Y$ such that $e(X)=0$, $\max_{y\in Y}d(y)\le \varepsilon^2m/4$, and $X$ has at least $d$ huge vertices or a unique huge vertex then
Conjecture~\ref{conj1} holds.

\begin{prop} \label{00001}
Let $d\geq 4$ be an integer and $\varepsilon>0$ be a real.
Let $D$ be a digraph with $m$ arcs and minimum outdegree at least $d$. Let $X,Y$ be a partition of $V(D)$
with $e(X)=0$ and $\max_{y\in Y}d(y)\le \varepsilon^2m/4$. Let
$\theta=\min\{|\theta(X_1,X_2)|: X_1,X_2  \mbox{ is a partition 
  of } X\}$ and  $X'=\{x\in X: s(x)\ge \theta\}$.  
  Suppose $|X'|\ge d$. Then $V(D)$ admits a partition $V_1,V_2$ such that 
$\min\{e(V_1,V_2),e(V_2, V_1)\}\geq \left(\frac{d-1}{2(2d-1)}-\varepsilon\right)m$.
\end{prop}

\begin{proof}  Suppose the desired partition $V_1, V_2$ does not exist. By (2) of Lemma~\ref{property1},
let $X'=\{v_1, \cdots, v_{2k+1}\}$. Then  $2k+1\ge d\ge 4$ by assumption. Let
$\Delta_i=s(v_i)$ for $i\in [2k+1]$ and assume, without loss of
generality,   $\Delta_1\ge \Delta _2\ge \ldots \ge \Delta_{2k+1}$. 
Let $X_1,X_2$ be a partition of $X$ such that
$\theta(X_1,X_2)=\theta$. Then, by  (1) and (2) of Lemma~\ref{property2},
we have
\begin{align*}
0> & d\left(\sum_{j=1}^{k}\Delta_{j}+g\right)-(d-1)\sum_{j=k+1}^{2k+1}\Delta_{j}
+b+m_2/2\\
  > & d\left(\sum_{j=1}^{k-1}\Delta_{j}+\Delta_{2k}\right)-(d-1)\left(\sum_{j=k}^{2k-1}\Delta_{j}+\Delta_{2k+1}\right)
 +\frac{d^2+2d-1}{d-1}\sum_{j=k}^{2k-1}\Delta_{j}\\
& -d\left(\sum_{j=1}^{k-1}\Delta_j+\Delta_{2k}+\Delta_{2k+1}\right)\\
 = & \frac{4d-2}{d-1}\sum_{j=k}^{2k-1}\Delta_{j}-(2d-1)\Delta_{2k+1}\\
\geq & \frac{k(4d-2)}{d-1}\Delta_{2k+1}-(2d-1)\Delta_{2k+1} \\
= & \frac{(2d-1)(2k+1-d)}{d-1}\Delta_{2k+1},
\end{align*}
This  is a contradiction, as $2k+1\geq d$.
\end{proof}

\medskip

{\it Remark}.  The requirement $e(X)=0$ in Proposition \label{1} can be replaced by $e(X)=o(m)$.

\medskip

Next,  we show that if $V(D)$ admits a partition $X,Y$ such that $|X|=o(|V(D)|)$, $\max_{y\in Y}d(y)\le \varepsilon^2m/4$,  and $D$ has a unique huge vertex in $X$ then
the conclusion of Conjecture~\ref{conj1} holds. For this, we need
another concept  introduced by Lee, Loh, and Sudakov \cite{Lee2013},
and we use the result of  Lu, Wang, and Yu in \cite{lu2015} to give its definition. We say that a connected graph is {\it tight}
if all its blocks are odd cliques.
If a disconnected graph $G$ is the underlying graph of a
digraph $D$, the {\it tight components} of $D$ are the
components of $G$ that are tight. (The $underlying \ graph$ of $D$ is obtained from $D$ by ignoring arc orientations and
removing redundant parallel edges.)
 For a tight component $T$ of $D$, we say $T$ is $essential$ if
 $D[V(T)]$, the subgraph of $D$ induced by $V(T)$, does not contain any parallel arcs in opposite directions. Recently, Hou, Li, and Wu \cite{Hou2020} proved the following.

\begin{lem}[Hou, Li, and Wu \cite{Hou2020}]\label{lem4}
For any positive constants $C$ and $\varepsilon$, there exist $\gamma$, $n_{0}>0$ for which the following holds.
Let $D$ be a digraph with $n\geq n_{0}$ vertices and at most $Cn$ arcs. Suppose $X\subseteq V(D)$ is a set
of at most $\gamma n$ vertices and $X_1,X_2$ is a partition of $X$.
Let $Y=V(D)\setminus X$ and let $\tau$ be the number of essential tight components in $D[Y]$.
If every vertex in $Y$ has degree at most
$\gamma n$ in $D$, then there is a bipartition
$V(D)=V_{1}\cup V_{2}$ with $X_{i}\subseteq V_{i}$ for $i=1,2$ such that
$$e(V_{1},V_{2})\geq e(X_{1},X_{2})+\frac{e(X_{1},Y)+e(Y,X_{2})}{2}+\frac{e(Y)}{4}+\frac{n-\tau}{8}-\varepsilon n,$$
$$e(V_{2},V_{1})\geq e(X_{2},X_{1})+\frac{e(X_{2},Y)+e(Y,X_{1})}{2}+\frac{e(Y)}{4}+\frac{n-\tau}{8}-\varepsilon n.$$
\end{lem}

\begin{prop}\label{huge-1}
Let $d\ge 4$ be an integer and let $C, \varepsilon$ be positive reals. Let $D$ be a digraph with $n$ vertices, $m\le Cn$ arcs, and minimum outdegree at least $d$. Then there exists $\gamma$ with $0<\gamma<\varepsilon$ such that the following holds:
Let $X,Y$ be a partition of $V(D)$ with $|X|\le \gamma n$, $e(X)=0$,
and $\max_{y\in Y}d(y)\le \gamma n$. Let
$\theta=\min\{|\theta(X_1,X_2)|: X_1,X_2 \mbox{ is a partition 
  of } X\}$,  $X'=\{x\in X: s(x)\ge \theta\}$, $g=\sum_{x\in
  X\setminus X'}s(x)$, and  $2b=\sum_{x\in X}\big(d(x)-s(x)\big)$.  
Then the following statements hold.

\begin{itemize}
\item [$(1)$]  $\tau\le (n+2g+2b)/(2d-2|X'|+1)$,  where $\tau$ is the number of essential tight components
in $D[Y]$.  

\item [$(2)$] Suppose $|X'|=1$. Then $V(D)$ admits a partition $V_1,V_2$ such that 
$\min\{e(V_1,V_2),e(V_2, V_1)\}\geq \left(\frac{d-1}{2(2d-1)}-\varepsilon\right)m.$
\end{itemize}
\end{prop}

\begin{proof}


First, we prove (1). Let $\alpha:=|X'|$.  For $i=1, 3, \ldots, 2d-2\alpha-1$, let $\tau_i$ be the number of essential tight components of order $i$;  and let $\tau'$ be the
number of essential tight components of order at least $2d-2\alpha+1$. Then
\begin{equation}\label{eq:i}
\tau_1+3\tau_3+\cdots+(2d-2\alpha-1)\tau_{2d-2\alpha-1}+(2d-2\alpha+1)\tau'\leq n.
\end{equation}
For each essential tight
component $D_i$ of order $i$, we see that $e(D_i)\le i(i-1)/2$ and $e(D_i,X')\le \alpha i$. Thus, since the outdegree of $D$ is
at least $d$, we see that $e(D_i,X\setminus X')\ge d i -\alpha i -i(i-1)/2$. Viewing $ d i -\alpha i -i(i-1)/2$ as a function of $i$
over the interval $[1, 2d-2\alpha]$, we see
that it achieves its minimum at $i=1$ (as well as at $i=2d-2\alpha$). Hence, $e(D_i,X\setminus X')\ge d-\alpha$ for $i\in [2d-2\alpha]$.
Thus, $e(Y, X\setminus X')\ge \sum_{i=1}^{d-\alpha}(d-\alpha)\tau_{2i-1}$. On the other hand, we have
$e(Y, X\setminus X')\le g+b$. Hence,
\begin{equation}\label{eq:j}
\sum_{i=1}^{d-\alpha} (d-\alpha) \tau_{2i-1} \leq g+b.
\end{equation}
Multiplying \eqref{eq:j} by 2 and adding the resulting inequality  to
\eqref{eq:i}, we derive that
$(2d-2\alpha+1)\tau \leq n+2g+2b$, completing the proof of (1).

\medskip
To prove (2), let $X'=\{v_0\}$ and
let $\Delta=s(v_0)$.  Let $X_1, X_2$ be the partition of $X$ such that $v_0$ is the only $(X_1,X_2)$-forward vertex. Then  $m_f(X_1,X_2)=\Delta+b$ and $m_b(X_1,X_2)=g+b$.
By Lemma \ref{lem4} (with $p=1/2$), there is a bipartition $V_1,V_2$ of $V(D)$ such that $X_i\subseteq V_i$ for $i\in [2]$ and
\begin{align*}
  &  \min \{e(V_1,V_2), e(V_2,V_2)\}\\
  \ge & \frac 12 \min\left\{e(X_1,Y)+e(Y,X_2), e(X_2,Y)+e(Y,X_1)\right\}+e(Y)/4+(n-\tau)/8-\varepsilon n\\
  = & (m-\theta)/4+(n-\tau)/8-\varepsilon n.
\end{align*}






Hence,
\begin{align*}
&\min \{e(V_1,V_2), e(V_2,V_2)\}
  -\left(\frac{d-1}{2(2d-1)}-\varepsilon \right)m\\
\ge & (m-\theta)/4+(n-\tau)/8-\varepsilon
      n-\left(\frac{d-1}{2(2d-1)}-\varepsilon \right)m\\
\ge & \frac{1}{4}\left( \frac{m}{2d-1}+\frac{n}{2}+4(d-1)\varepsilon n-\theta-\frac{\tau}{2}\right) \quad (\mbox{since $m\ge dn$}).
\end{align*}
Since the minimum outdegree of $D$ at least $d$ and
$|Y|=n-|X|>n-\varepsilon n$, we have $$m\geq b+d|Y|\geq b+dn-d\varepsilon n.$$
Therefore,
\begin{align*}
    &\frac{m}{2d-1}+\frac{n}{2}+4(d-1)\varepsilon n-\theta-\frac{\tau}{2} \notag \\
\geq & \frac{b+dn-d\varepsilon n}{2d-1}+\frac{n}{2}+4(d-1)\varepsilon n-\theta-\frac{n+2g+2b}{2(2d-1)} \quad (\mbox{by (1)})\\
\geq & \frac{1}{2(2d-1)}\Big( 2b+2dn+(2d-1)n-(4d-2)\theta-n-2g-2b \Big)\\
= & \frac{1}{2(2d-1)}\Big((4d-2)n-(4d-2)\theta-2g\Big)\\
\ge &0 \quad (\mbox{since $n\ge |Y|\ge \theta+g$ by Lemma \ref{lem5}}).
\end{align*}
So $V_1,V_2$ gives the desired partition for (2).
\end{proof}

We now use Lemma~\ref{lem4} to refine (1) of Lemma~\ref{property2} for the case
when there are only three huge vertices.

\begin{lem}\label{3huge-vertices}

Let $D$ be a digraph with $m$ arcs and minimum outdegree $d\ge 4$, and
let $X,Y$ be a partition of $V(D)$ with $e(X)=0$.  Let
$\theta=\min\{|\theta(X_1,X_2)|: X_1,X_2 \mbox{ is a partition of } X\}$,  $X'=\{x\in X: s(x)\ge \theta\}$,
$g=\sum_{x\in X\setminus X'}s(x)$, $2b=\sum_{x\in X}\big(d(x)-s(x)\big)$, and $\tau$ the number of essential tight components in $D[Y]$.
Let $\varepsilon>0$ and assume that $\max_{y\in Y}d(y)\le \varepsilon^2m/4$.
Suppose $|X'|=3$. Then there exists a partition $V_1,V_2$ of $V(D)$
such that $\min\{e(V_1,V_2),e(V_2,V_1)\}\ge (3/14-\varepsilon) m$;  or
if we write  $X'=\{v_1,v_2,v_3\}$ with $s(v_1)\ge
s(v_2)\ge s(v_3)$ and $\Delta_j=s(v_j)$ for $j\in [3]$  then
  $$b <3 \Delta_2+3\Delta_3-4\Delta_1-4g-m_2/2-7(|V(D)|-\tau)/4.$$
\end{lem}

\begin{proof}  We consider the partition $X_1,X_2$ of $X$ such
  that $\{v_{1}\}\cup (X\setminus X')$ is the set of $(X_1,X_2)$-forward vertices,
 and  $\{v_{2},v_{3}\}$ is the set of
 $(X_1,X_2)$-backward vertices.  Then $m_f(X_1,X_2)=\Delta_{1}+g+b$ and
$m_b(X_1,X_2)=\Delta_{2}+\Delta_3+b$. Note that
\begin{equation}\label{eq-1}
m=m_1+m_2=\sum_{j=1}^{3}\Delta_{j}+g+2b+m_2.
\end{equation}

Let $n:=|V(D)|$. By applying Lemma \ref{lem4} with $p=1/2$, there is a partition $V(D)=V_1\cup V_2$ such that
$X_i\subseteq V_i$ for $i\in [2]$, and
\begin{equation}\label{eq-2}
\left\{
\begin{aligned}
 e(V_{1},V_{2}) & \ge (\Delta_{1}+g+b)/2+ m_2/4 +(n-\tau)/8-\varepsilon n, \\
 e(V_{2},V_{1}) & \ge  (\Delta_{2}+\Delta_3+b)/2+m_2/4+(n-\tau)/8-\varepsilon n.
\end{aligned}
\right.
\end{equation}
Define
\begin{itemize}
   \item [] $f(1/2,X_1,X_2)= 4\Delta_1-3\Delta_2-3\Delta_3+4g+b+m_2/2+7(n-\tau)/4$, and
   \item [] $h(1/2,X_1,X_2)= 4\Delta_2+4\Delta_3-3\Delta_1-3g+b+m_2/2+7(n-\tau)/4$.
\end{itemize}


By (3) of Lemma~\ref{property1}, we may assume
$\Delta_2+\Delta_3-\Delta_1\ge g+\theta$; for otherwise the desired
partition of $V(D)$ exists. Hence, $ h(1/2,X_1,X_2)>0$.

By \eqref{eq-1} and \eqref{eq-2}, we have
\begin{eqnarray*}
& & e(V_{1},V_{2})-(3m/14-\varepsilon m)\\
& \geq &  (\Delta_{1}+g+b)/2+m_2/4+(n-\tau)/8-(3/14)\left(\sum_{j=1}^{3}\Delta_{j}+g+2b+m_2\right)\\
& = & (1/14) f(1/2,X_1,X_2),
\end{eqnarray*}
and
\begin{eqnarray*}
& & e(V_{2},V_{1})-(3m/14-\varepsilon m)\\
& \geq & (\Delta_{2}+\Delta_3+b)/2+m_2/4+(n-\tau)/8-(3/14)\left(\sum_{j=1}^{3}\Delta_{j}+g+2b+m_2\right)\\
& = & (1/14) h(1/2,X_1,X_2).
\end{eqnarray*}

Thus, if  $f(1/2,X_1,X_2)\ge 0$ then, since  $ h(1/2,X_1,X_2)>0$, the
desired partition of $V(D)$ exists. So, we may assume $f(1/2,X_1,X_2)<0$ which implies the desired inequality.
\end{proof}

\section{Proof of Theorem \ref{thm1}}

In this section, we prove Theorem~\ref{thm1}, by using Propositions
\ref{00001} and \ref{huge-1} and Lemma~\ref{3huge-vertices} and by choosing $X$ to consist of vertices
of degree at most $n^{3/4}$. Our proof is much simpler when applied to
the cases $d=2$ and $d=3$, and gives the results of Lee, Loh, and Sudakov \cite{Lee2014} that
Conjecture \ref{conj1} is true for these cases.

\medskip

\noindent {\it Proof of Theorem~\ref{thm1}}.  Let $D$ be a digraph with $n$ vertices and $m$ arcs, and assume that the minimum outdegree of $D$ is at least $4$.   We wish to find a partition $V(D)=V_1\cup V_2$, such
that $\min\{e(V_1,V_2),e(V_2,V_1)\}\geq \big(3/14+o(1)\big)m$.  We may assume  that $n$ is sufficiently large so that all
lemmas in the previous sections can be applied. We claim that

\begin{itemize}
\item [(1)] $m\ge 4n$ and we may assume $m< 128\cdot  7^2 n$.
\end{itemize}
Since $D$ has minimum outdegree at least $4$, we have $m\ge 4n$. Now suppose  $m\geq 128\cdot 7^2 n$.
Then applying Corollary \ref{lem2} with $\varepsilon=1/28$ we obtain a partition $V(D)=V_1\cup V_2$ such that
\[
min\{e(V_1,V_2),e(V_2,V_1)\}\geq \left(1/4-1/28\right)m= 3m/14.
\]
So we may assume $m< 128 \cdot 7^2 n$. $\Box$

\medskip

Consider the partition $X,Y$ of $V(D)$ such that  $X=\{v\in
V(D):d(v)\geq n^{3/4}\}$ and $Y=V(D)\setminus X$. Then, by (1),
\[
|X|\cdot n^{3/4}\leq \sum_{v\in X}d(v)\leq \sum_{v\in V(D)}d(v)=2m< 256\cdot 7^2 n.
\]
Hence,  $|X|=O(n^{1/4})$ and,  thus,  $e(X)\leq|X|^{2}=O(n^{1/2})=o(m)$. Therefore, we may assume
\begin{itemize}
\item [(2)] $e(X)=0$.
\end{itemize}


Let $\theta=\min\{|\theta(S,T)|: S,T \mbox{ is a partition  of } X\}$.
Let  $X_1,X_2$ be a partition of $X$ such that
$\theta(X_1,X_2)=\theta$, and let $X'=\{x\in X: s(x)\ge \theta\}$. If
$|X'|=1$ or $|X'|\ge 5$ then the desired partition exists by
Propositions~\ref{00001} and \ref{huge-1}. So
 let $X'=\{v_1,v_2,v_3\}$ and $\Delta_i=s(v_i)$ for $i\in [3]$ such
 that $\Delta_1\ge \Delta_2\ge \Delta_3$.
Since $|X|=o(n)$,
$$\Delta_1+\Delta_2+\Delta_3+g+b+m_2\ge \sum_{y\in Y}d^+(y) \ge
4|Y|=4n-o(n).$$
Hence,  writing $t=2\Delta_1-(\Delta_2+\Delta_3)$, we have
\begin{equation}\label{5.2}
b\ge 4n-\Delta_1-\Delta_2-\Delta_3-g-m_2-o(n)=4n- 3\Delta_1-g-m_2+t-o(n),
\end{equation}
Next, we will derive bounds on $m_2,\Delta_1,t$ and $b$ in terms of
$n$ and $g$, so that we can use (4) of Lemma~\ref{property2}.

By Lemma~\ref{3huge-vertices},
$$b<3\Delta_2+3\Delta_3-4\Delta_1-4g-m_2/2-7(n-\tau)/4=2\Delta_1-3t-4g-m_2/2-7(n-\tau)/4, $$
which, combined with $\tau\leq (n+2g+2b)/3$ (by (1) of Proposition~\ref{huge-1}), imlpies
\begin{equation}\label{5.1}
b>7n+3m_2+17g-12\Delta_1+18t.
\end{equation}

We may assume that (2) and
(3) of Lemma~\ref{property2} hold; for, otherwise, by
Lemma~\ref{property2}, the desired partition of $V(D)$ exists. Thus,
\begin{equation}\label{5.3}
b>3g+3m_2/8-\Delta_1/3+4t,
\end{equation}
and
\begin{equation}\label{5.4}
b< \frac{28}{11}n-\frac{27}{11}\Delta_1+\frac{12}{11}g-\frac{9}{88}m_2+\frac{2}{11}t.
\end{equation}


Combining \eqref{5.4} with \eqref{5.2}, \eqref{5.1}, \eqref{5.3},
respectively, we obtain (by eliminating $b$)
\begin{equation}\label{5.6}
\frac{79}{8}m_2+23g>16n-6\Delta_1+9t-o(n),
\end{equation}
\begin{equation}\label{5.5}
\frac{273}{8}m_2+175g<105\Delta_1-49n-196t.
\end{equation}
\begin{equation}\label{5.7}
\frac{63}{4}m_2+63g<84n-70\Delta_1-126t.
\end{equation}
Noting that $\Delta_1\le n-|X|$ (by (2)) and $n$ is large, we see from \eqref{5.6}
that $\frac{79}{8}m_2+23g> 9.99n$; so
\begin{equation}\label{000}
m_2+2.31g> n
\end{equation}

We now combine \eqref{5.6} with \eqref{5.5} and
\eqref{5.7}, respectively,  and we get (by eliminating $m_2$)
\begin{equation}\label{5.8}
1.806t+0.759g<\Delta_1-0.829n,
\end{equation}
\begin{equation}\label{5.9}
2.322t+0.435g<0.968n-\Delta_1.
\end{equation}
From \eqref{5.8}, we have
\begin{equation}\label{001}
\Delta_1>0.829n.
\end{equation}
 Combining \eqref{5.8} and \eqref{5.9} to eliminate $\Delta_1$, we have
\begin{equation}\label{002}
3t+g<0.117n.
\end{equation}

From \eqref{5.4}, we have
$$b
<\frac{28}{11}n-\frac{27}{11}\Delta_1-\left(\frac{9}{88}m_2+\frac{21}{88}g\right)+\left(\frac{2}{11}t+\frac{117}{88}g\right).$$
Hence, by \eqref{000}, \eqref{001} and \eqref{002}, we have
\begin{equation}\label{003}
b  <\frac{28}{11}n-\frac{27}{11}\times 0.829n-\frac{9}{88}n+\frac{117}{88}\times 0.117n <0.564n.
\end{equation}

We wish to use (4) of Lemma~\ref{property2}. Note that
$$  2\Delta_2+2\Delta_3-3\Delta_1-3g-b+3m_2/14=\Delta_1+(3m_2/14+g/2)-b-(2t+7g/2).$$
So by \eqref{000}, \eqref{001}, \eqref{002}, and \eqref{003}, we see that
\begin{align*}
  & 2\Delta_2+2\Delta_3-3\Delta_1-3g-b+3m_2/14\\
   > & 0.829n+3n/14-0.564n-(7/2)\times  0.117n\\
  > & 0.069n\\
  > & 0.
\end{align*}

Hence, by (4) of Lemma \ref{property2}, we have $
6\Delta_1-3\Delta_2-3\Delta_3+2g-b+3m_2/14<0$. Thus,  
\begin{equation}\label{5.10}
b > 6\Delta_1-3\Delta_2-3\Delta_3+3m_2/14+2g=3m_2/14+2g+3t. 
\end{equation}
Combining \eqref{5.4} and \eqref{5.10} (by eliminating $b$), we have
$$\frac{195}{56}m_2+10g<28n-27\Delta_1-31t,$$
which, combined with \eqref{5.6} (by eliminating $m_2$), gives
\begin{equation}\label{5.12}
1.373t+0.075g<0.899n-\Delta_1,
\end{equation}
which implies $\Delta_1<0.899n$; so
by \eqref{5.2}, we have
\begin{equation}\label{005}
g+b+m_2>4n-3\Delta_1+t-o(n)>4n-3\times 0.9n=1.3n.
\end{equation}
Combining \eqref{5.8} and \eqref{5.12} (to eliminate $\Delta_1$), we derive
\begin{equation}\label{006}
3t+g<0.084n.
\end{equation}

Again by (4) of Lemma~\ref{property2}, we have
$$0  >  6\Delta_3-3\Delta_1-3\Delta_2-3g+b/3+3m_2/14 =
 3(m_2+b+g)/14+5b/42-(6t+2g+17g/14)+9(\Delta_1-\Delta_2).$$
Hence, noting that $b>3n/14$ (by \eqref{000} and \eqref{5.10}) and by \eqref{005} and \eqref{006},
$$0 >(3/14)\times 1.3n+(5/42) \times (3n/14)-(2+17/14))\times 0.084n>0.034n.$$
This is a contradiction,  completing the proof of Theorem~\ref{thm1}.

\medskip

\section{Concluding remarks}
We studied partitions of digraphs with minimum outdegree $d\ge 4$ and proved Conjecture \ref{conj1} in the case when $d=4$.
We used a typical approach  for finding a partition $V_1,V_2$ in a digraph $D$ that bounds $e(V_1,V_2)$ and $e(V_2,V_1)$ simultaneously:
Start with a partition $X,Y$ of $V(D)$ such that $X$ consists of
large degree vertices;  partition $X$ by
considering the ``huge'' vertices in $X$, those vertices with large
gap between their outdegree and their indegree; and randomly partition the
vertices in $Y$. Huge vertices play an important role in the process
for  obtaining the desired partition. For instance,
 we showed that Conjecture~\ref{conj1}
holds when there exists a partition of $V(D)$ for which  the number of
huge vertices is at least $d$ or exactly 1.
We hope that our work would shed light on how the set $V(D)$ should be
partitioned into $X,Y$ and how the set $X$ should be
partitioned.

In \cite{Lee2014}, Lee, Loh, and Sudakov point out that one needs  to
combine both Lemma $\ref{lem1}$ and Lemma \ref{lem4}
to prove Conjecture \ref{conj1}. They also remarked that a naive
combination is not adequate for $d\geq 4$ because of the following
example.  Let $D'$ be the digraph obtained from $K_{5,n-5}$ (with
$n>9$) by  orienting the edges so that
one vertex, say $v_{1}$, has outdegree $n-5$ and  four
vertices, say  $v_{2}, v_{3}, v_{4}, v_{5}$, each have indegree $n-5$.
Let $D$ be obtained from $D'$ by adding an arc directed from $v_i$ to $v_j$ for
each ordered pair $(i,j)$ with $1\le i\ne j\le 5$. Then the minimum
outdegree of $D$ is   4 and the number of arcs in $D$ is $m=5n-5$.
Let $X=\{v_1, \ldots, v_5\}$. If we partition $V(D)$ to $X$ (consisting of large degree vertices)
and $Y=V(D)\setminus X$ (consisting of small degree vertices), then
$X$  is the set of huge vertices.  One can check  that
$X_1=\{v_2 \}$ and $X_2=\{v_1,v_3,v_4,v_5 \}$ form  a partition of $X$
with  minimum gap.  However, for any partition $Y_1,Y_2$ of $Y$, we see that
 $e(X_2\cup Y_{2},X_1\cup Y_{1})=n-5+4=n-1=m/5$, which is smaller than $3m/14$.
What this means is that one need to consider different partitions of the
set  of huge vertices. In this paper, we have managed to prove Conjecture~\ref{conj1} in the case when $d=4$ by
carefully partitioning the huge vertices.

For digraphs with minimum outdegree $d\ge 5$, new ideas seem needed (in addition to better
partitioning the huge vertices), as shown by the following example.
Let $D'$ be the digraph obtained from $K_{3,n-3}$ (with $n>900$) by
orienting all edges from the part $Y$ of
size $n-3$ to the part $X$ of size 3. Let $D$ be obtained from $D'$ by
adding
six arcs directed from each vertex in $X$ to  6 vertices in $Y$ (so
that no two arcs get directed towards the same vertex in $Y$),
and adding a 3-out-regular graph on $Y$.
Hence, the minimum outdegree of $D$ is  $6$, and $X$ is the set of
huge vertices with respect to the partition $X,Y$.     It is not difficult to verify that for any value $p$
and any partition $X_1,X_2$ of $X$, we have
$f(p,X_1,X_2)<0$ and $h(p,X_1,X_2)<0$ (see Section 3). Therefore,  one needs to better
partition $V(D)\setminus X$ in order to achieve the bound in Conjecture \ref{conj1}.



\end{document}